\documentclass[11pt]{amsart}

\usepackage{amsmath, amsthm, amssymb, amsfonts}
\usepackage{fullpage, fancyhdr, qtree, tipa, url, subfigure, enumitem, setspace, graphicx}

\newtheorem{thm}{Theorem}[section]
\newtheorem{lem}[thm]{Lemma}

\newtheorem{prop}[thm]{Proposition}
\newtheorem{df}[thm]{Definition}
\newtheorem{conj}[thm]{Conjecture}
\newtheorem{cor}[thm]{Corollary}
\newtheorem{rem}[thm]{Remark}
\newtheorem{exam}[thm]{Example}

\numberwithin{equation}{section}

\title{On Orbits of Order Ideals of Minuscule Posets}
\author{David B Rush \and XiaoLin Shi}
\address{Department of Mathematics, Massachusetts Institute of Technology, Cambridge, MA}
\email{dbr@mit.edu}
\address{Department of Mathematics, Massachusetts Institute of Technology, Cambridge, MA}
\email{dannyshi@mit.edu}
\date{\today}

\begin{document}

\begin{abstract}

An action on order ideals of posets considered by Fon-Der-Flaass is analyzed in the case of posets arising from minuscule representations of complex simple Lie algebras.  For these minuscule posets, it is shown that the Fon-Der-Flaass action exhibits the cyclic sieving phenomenon, as defined by Reiner, Stanton, and White.  A uniform proof is given by investigation of a bijection due to Stembridge between order ideals of minuscule posets and fully commutative Weyl group elements.  This bijection is proven to be equivariant with respect to a conjugate of the Fon-Der-Flaass action and an arbitrary Coxeter element.  

If $P$ is a minuscule poset, it is shown that the Fon-Der-Flaass action on order ideals of the Cartesian product $P \times [2]$ also exhibits the cyclic sieving phenomenon, only the proof is by appeal to the classification of minuscule posets and is not uniform.  

\end{abstract}

\maketitle

\section{Introduction} 

The Fon-Der-Flaass action on order ideals of a poset has been the subject of extensive study since it was introduced in its original form on hypergraphs by Duchet in 1974 \cite{Duchet}.  In this article, we identify a disparate collection of posets characterized by properties from representation theory  -- the minuscule posets --  that exhibits consistent behavior under the Fon-Der-Flaass action.  We illustrate the commonality via the cyclic sieving phenomenon of Reiner-Stanton-White \cite{Reiner}, which provides a unifying framework for organizing combinatorial data on orbits arising from cyclic actions.  

If $P$ is a poset, and $J(P)$ is the set of order ideals of $P$, partially ordered by inclusion, the Fon-Der-Flaass action $\Psi$ maps an order ideal $I \in J(P)$ to the order ideal $\Psi(I)$ whose maximal elements are the minimal elements of $P \setminus I$.  Since $\Psi$ is invertible, it generates a cyclic group $\langle \Psi \rangle$ acting on $J(P)$, but the orbit structure is not immediately apparent.  

In \cite{Reiner}, Reiner, Stanton, and White observed many situations in which the orbit structure of the action of a cyclic group $\langle c \rangle$ on a finite set $X$ may be predicted by a polynomial $X(q) \in \mathbb{Z}[q]$. \\

\noindent \textbf{Definition.}  The triple $(X,X(q),\langle c \rangle)$ exhibits the \textit{cyclic sieving phenomenon} if, for any integer $d$, the number of elements $x$ in $X$ fixed by $c^d$ is obtained 
by evaluating $X(q)$ at $q=\zeta^d$, where $n$ is the order of $c$ on $X$ and $\zeta$ is any primitive $n^{th}$ root of unity.  \\

In the case when $X = J(P)$ and $c$ is the Fon-Der-Flaass action, the natural generating function to consider is the rank-generating function for $J(P)$, which we denote by $J(P; q)$.  Here the rank of an order ideal $I \in J(P)$ is given by the cardinality $|I|$ (so that $J(P;q) := \sum_{I \in J(P)} q^{|I|}$).  

The minuscule posets are a class of posets arising in the representation theory of Lie algebras that enjoy some astonishing combinatorial properties.  We give some background.  

Let $\mathfrak{g}$ be a complex simple Lie algebra with Weyl group $W$ and weight lattice $\Lambda$.  There is a natural partial order on $\Lambda$ called the \textit{root order} in which one weight $\mu$ is considered to be smaller than another weight $\omega$ if the difference $\omega-\mu$ may be expressed as a positive linear combination of simple roots.  If $\lambda \in \Lambda$ is dominant and the only weights occuring in the irreducible highest weight representation $V^\lambda$ are the weights in the $W$-orbit $W \lambda$, then $\lambda$ is called \textit{minuscule}, and the restriction of the root order to the set of weights $W \lambda$ (which is called the \textit{weight poset}) has two alternate descriptions:  

\begin{itemize}
\item
Let $W_J$ be the maximal parabolic subgroup of $W$ stabilizing $\lambda$, and let $W^J$ be the set of minimum-length coset representatives for the parabolic quotient $W/W_J$.  Then there is a natural bijection
$$
\begin{array}{rcl}
W^J & \longrightarrow & W \lambda \\
w & \longmapsto & w_0w\lambda \\
\end{array}
$$
(where $w_0$ denotes the longest element of $W$), and this map is an isomorphism of posets between the \textit{strong Bruhat order} on $W$ restricted to $W^J$ and the root order on $W \lambda$.
\item
Let $P$ be the poset of join-irreducible elements of the root order on $W \lambda$.  Then $P$ is called the \textit{minuscule poset for }$\lambda$, $P$ is ranked, and there is an isomorphism of posets between the weight poset and $J(P)$.  
\end{itemize}

If $P$ is minuscule, Proctor showed (\cite{Proctor}, Theorem 6) that $P$ enjoys what Stanley calls the \textit{Gaussian} property (cf. \cite{Stanleyec}, Exercise 25): There exists a function $f: P \rightarrow \mathbb{Z}$ such that, for all positive integers $m$, 
$$
J(P \times [m]; q) = \prod_{p \in P} \frac{1 - q^{m+f(p)+1}}{1 - q^{f(p)+1}}.
$$

This may be verified case-by-case, but it follows uniformly from the \textit{standard monomial theory} of Lakshmibai, Musili, and Seshadri, as is shown in \cite{Proctor}.  Furthermore, all Gaussian posets are ranked, and if $P$ is Gaussian, we may take $f$ to be the rank function of $P$.  

Thus, for all positive integers $m$, we are led to consider the triple $(X, X(q), \langle \Psi \rangle)$, where $X = J(P \times [m])$, $X(q) = J(P \times [m];q)$, and $P$ is any minuscule poset.  We are at last ready to state the first two of our main results, answering a question of Reiner.

\begin{thm}\label{main1}
Let $P$ be a minuscule poset.  If $m=1$, $(X, X(q), \langle \Psi \rangle)$ exhibits the cyclic sieving phenomenon.  
\end{thm}

\begin{thm}\label{main2}
Let $P$ be a minuscule poset.  If $m=2$, $(X, X(q), \langle \Psi \rangle)$ exhibits the cyclic sieving phenomenon.  
\end{thm}
 
It turns out that the claim analogous to Theorems~\ref{main1} and \ref{main2} is false for $m=3$; computations performed by Kevin Dilks\footnote{Computed using code in the computer algebra package {\tt Maple}.  The authors also thank Dilks for allowing them the use of his code for subsequent computations.} reveal that when $m=3$ and $P$ is the minuscule poset $[3] \times [3]$, the triple $(X, X(q), \langle \Psi \rangle)$ does not exhibit the cyclic sieving phenomenon.  However, if $P$ belongs to the third infinite family of minuscule posets (see the classification at the end of the introduction), the same triple exhibits the cyclic sieving phenomenon for all positive integers $m$.  This was proved in our original REU report \cite{RushShi} but is omitted here.  The rest of this introduction is devoted to a discussion of Theorems~\ref{main1} and \ref{main2} and a brief overview of our approach to their proofs.  

It should be noted that several special cases of Theorem~\ref{main1} already exist in the literature.  When $P$ arises from a Lie algebra with root system of type $A$, for instance, Theorem~\ref{main1} reduces to a result of Stanley in \cite{Stanley} coupled with Theorem 1.1(b) in Reiner-Stanton-White \cite{Reiner}, and it is recorded as Theorem 6.1 by Striker and Williams in \cite{Jessica}.  The case when the root system is of type $B$ turns out to be handled almost identically, and it is recorded as Corollary 6.3 in \cite{Jessica}.  That being said, our theorem is a generalization of these results, and, in relating Theorem~\ref{main1} to a known cyclic sieving phenomenon for finite Coxeter groups (Theorem 1.6 in \cite{Reiner}), we expose the Fon-Der-Flaass action to new algebraic lines of attack.

If $P$ is a finite poset, it is shown by Cameron and Fon-Der-Flaass in \cite{CameronFonderflaass} that the Fon-Der-Flaass action $\Psi$ may be expressed as a product of the involutive generators $\lbrace t_p \rbrace_{p \in P}$ for a larger group acting on the poset of order ideals $J(P)$.  For all $p \in P$ and $I \in J(P)$, $t_p(I)$ is obtained by \textit{toggling $I$ at $p$}, so that $t_p(I)$ is either the symmetric difference $I \Delta \lbrace p \rbrace$, if this forms an order ideal, or just $I$, otherwise.  In \cite{Jessica}, Striker and Williams named this group the \textit{toggle group}. 

On the other hand, there is a natural labeling of the elements of a minuscule poset $P$ by the Coxeter generators $S = \lbrace s_1, s_2, \ldots, s_n \rbrace$ for the Weyl group $W$, which is given by Stembridge in \cite{Stembridge2}.  In particular, if $P$ is a minuscule poset, there exists a labeling of $P$ such that the linear extensions of the labeled poset (which is called a \textit{minuscule heap}) index the reduced words for the fully commutative element of $W$ representing the topmost coset $w_0W_J$.  This labeling is illustrated in Figure~\ref{ExampleA4} (as well as in Figures~\ref{figa}, \ref{figb}, \ref{figc}, \ref{figd1}, \ref{fige6}, and \ref{fige7} of the Appendix) and explained more thoroughly in section 5.  It has the following important properties.  

First, it realizes the poset isomorphism $J(P) \cong W^J$ explicitly.  Given an order ideal $I \in J(P)$ and a linear extension $(x_1, x_2, \ldots, x_t)$ of the partial order restricted to the elements of $I$, if the corresponding sequence of labels is $(i_1, i_2, \ldots, i_t)$, define $\phi(I)$ to be $s_{i_t} \cdots s_{i_2} s_{i_1}$.  Then the map $\phi: J(P) \rightarrow W^J$ is an order-preserving bijection.  

Second, it indicates a correspondence between Coxeter elements in $W$ and sequences of toggles in $G(P)$: The choice of a linear ordering on the Coxeter generators $S=(s_{i_1},\ldots,s_{i_n})$ yields a choice of the following. 
\begin{itemize}
\item
An element $t_{(i_1,\ldots,i_n)}$ in the toggle group that executes the following sequence of toggles: first toggle at all elements of $P$ labeled by $s_{i_n}$, in any order; then toggle at all elements of $P$ labeled by $s_{i_{n-1}}$, in any order;...; then toggle at all elements of $P$ labeled by $s_{i_2}$, and, finally, toggle at all elements of $P$ labeled by $s_{i_1}$, and
\item
A \textit{Coxeter element} $c=s_{i_1} s_{i_2} \cdots s_{i_{n-1}} s_{i_n}$ in the Weyl group, which acts on cosets $W/W_J$ by left translation (i.e., $c(wW_J)=cwW_J$), and thus also acts on $W^J$.
\end{itemize}

The theorems that reduce Theorem~\ref{main1} to the cyclic sieving result of Reiner-Stanton-White \cite{Reiner} are as follows.  

\begin{thm}\label{secondary1}
For any minuscule poset $P$ and any ordering of  $S=(s_{i_1},\ldots,s_{i_n})$, the actions $\Psi$ and $t_{(i_1,\ldots,i_n)}$ are conjugate in $G(P)$.  
\end{thm}

\begin{thm}\label{secondary2}
For any minuscule poset $P$ and any ordering of  $S=(s_{i_1},\ldots,s_{i_n})$, if \\
$\phi: J(P) \rightarrow W^J$ is the isomorphism described above, then the following diagram is commutative:
$$
\begin{array}{rcl}
J(P) &\overset{\phi}{\longrightarrow}& W^J\\
t_{(i_1,\ldots,i_n)}\downarrow& &c\downarrow \\
J(P) &\overset{\phi}{\longrightarrow}& W^J.\\
\end{array}
$$
\end{thm}
\noindent

To see that these theorems suffice to demonstrate Theorem~\ref{main1}, we quote Theorem 1.6 from \cite{Reiner}.

\begin{thm} \label{Victheorem}

Let $W$ be a finite Coxeter group; let $S$ be the set of Coxeter generators, and let $J$ be a subset of $S$.  Let $W^J$ be the set of minimum-length coset representatives, and let $W^J(q) = \sum_{w \in W^J} q^{l(w)}$, where $l(w)$ denotes the length of $w$.  If $c \in W$ is a regular element in the sense of Springer \cite{Springer}, then $(W^J, W^J(q), \langle c \rangle)$ exhibits the cyclic sieving phenomenon.  

\end{thm}

In Theorem~\ref{Victheorem}, if $W^J$ is a distributive lattice, then the length function $l$ also serves as a rank function, so $W^J(q)$ is the rank-generating function.  Furthermore, if $c$ is a Coxeter element of $W$, then $c \in W$ is regular (cf. \cite{Springer}).  

The proofs of Theorem~\ref{secondary1} and Theorem~\ref{secondary2} are carried out in section 6; sections 2, 3, 4, and 5 provide the requisite background.  We did not manage to adapt the techniques developed in these sections for the proof of Theorem~\ref{main2}, so in sections 7-11 we adopt a less theoretical approach, suppressing most of the details.  Full proofs may still be found in the REU report \cite{RushShi}.  In section 7, we review ordinary and symmetric plane partitions, which provide a convenient framework for analyzing order ideals of $P \times [2]$.  Then sections 8, 9, and 10 cover the cases corresponding to the first, second, and third infinite families, respectively.  The claim of Theorem~\ref{main2} for the two exceptional cases is checked by computer in section 11, using the software developed by Dilks.  While the proofs of the results which we assemble into Theorem~\ref{main2} are purely combinatorial, we would like to see a uniform resolution of this problem that draws upon the more algebraic techniques of sections 2-6.  That, in particular, may seem like a tall order, but it should be noted that, for minuscule posets $P$, Stembridge found an instance of the $q=-1$ phenomenon (a special case of the cyclic sieving phenomenon for actions of order 2) that holds uniformly for all Cartesian products $P \times [m]$ in \cite{Stembridge1}.  Thus, even though the situation in the case of general cyclic sieving is considerably more complicated, there may still be reason to be optimistic. 

We close the introduction with a description of the three infinite families and two exceptional cases of minuscule posets and the root systems associated to the Lie algebras from which they arise.  Pictures may be found in the appendix.  The following facts are well-known (cf. for instance, \cite{Stembridge1}).  

\begin{itemize}
\item For the root systems of the form $A_n$, there are $n$ possible minuscule weights, which lead to $n$ associated minuscule posets, namely, all those posets of the form $[j] \times [n+1-j]$ such that $1 \leq j \leq n$.  Posets of this form are considered to comprise the \textit{first infinite family}.  Examples are depicted in Figure~\ref{figa}, parts (b), (c), (d), and (e).   

\item For the root systems of the form $B_n$, there is 1 possible minuscule weight, which leads to 1 associated minuscule poset, namely, $[n] \times [n] / S_2$.  Posets of this form are considered to comprise the \textit{second infinite family}.  An example is depicted in Figure~\ref{figb}, part (b).

\item For the root systems of the form $C_n$, there is 1 possible minuscule weight, which leads to 1 associated minuscule poset, namely, $[2n-1]$.  Posets of this form already belong to the first infinite family.  An example is depicted in Figure~\ref{figc}, part (b).  

\item For the root systems of the form $D_n$, there are 3 possible minuscule weights, which, because two of the minuscule weights lead to the same minuscule poset, only lead to 2 associated minuscule posets, namely, $[n-1] \times [n-1] / S_2$ and $J^{n-3}([2] \times [2])$.  Posets of the latter form are considered to comprise the \textit{third infinite family} (it should be clear that posets of the former form already belong to the second infinite family).  An example is depicted in Figure~\ref{figd1}, part (d).    

\item For the root system $E_6$, there are 2 possible minuscule weights, which, because both minuscule weights lead to the same minuscule poset, only lead to 1 associated minuscule poset, namely, $J^2([2] \times [3])$.  This poset is called the \textit{first exceptional case}.  It is depicted in Figure~\ref{fige6}, part (b). 

\item For the root system $E_7$, there is 1 possible minuscule weight, which leads to 1 associated minuscule poset, namely, $J^3([2] \times [3])$.  This poset is called the \textit{second exceptional case}.  It is depicted in Figure~\ref{fige7}, part (b). 

\end{itemize}

No other root systems admit minuscule weights.  

\section{The Fon-Der-Flaass Action}
In this section, we introduce and analyze the Fon-Der-Flaass action.  This action was introduced by Duchet \cite{Duchet} and first studied in its present form by Brouwer and Schrijver \cite{Brouwer}, but it was the late Dmitry Fon-Der-Flaass who first made substantial progress in the case of products of chains \cite{FonDerFlaass}, and it was this work that brought the action to our attention.  The action has no accepted name in the literature, so we are honored to dedicate it to his memory.  

Let $P = (X, <)$ be a partially ordered set, and let $J(P)$ be the set of order ideals of $P$, partially ordered by inclusion.  Following the notation of \cite{CameronFonderflaass}, for all order ideals $I \in J(P)$, let 
$$Z(I) = \lbrace x \in I : y > x \Longrightarrow y \notin I \rbrace,$$
and let 
$$U(I) = \lbrace x \notin I : y < x \Longrightarrow y \in I \rbrace.$$  
Then the \textit{Fon-Der-Flaass action}, which we denote by $\Psi$, is defined as follows.

\begin{df}\label{dfFonderflaass} \rm
For all $I \in J(P)$, $\Psi(I)$ is the unique order ideal satisfying $Z(\Psi(I)) = U(I)$. 
\end{df}
\begin{rem}\rm
It is clear from Definition~\ref{dfFonderflaass} that $\Psi$ permutes the order ideals of $P$. 
\end{rem}

\begin{figure}[htp]
\begin{center}
\includegraphics[scale = 0.8]{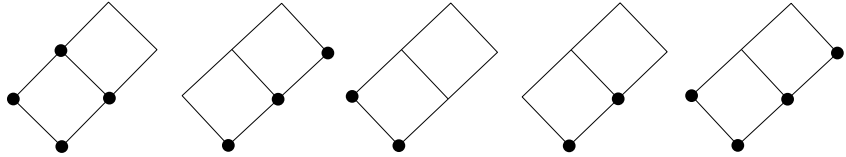} 
\caption{An orbit of order ideals under the Fon-Der-Flaass action}
\label{fig: Fon-Der-Flaass Operation}
\end{center}
\vspace{0in}
\end{figure}

This definition of the Fon-Der-Flaass action is global.  We now give an equivalent definition that decomposes it into a product of local actions, which are more easily understood.  Recall from the introduction that for all $p \in P$ and $I \in J(P)$, we let $t_p: J(P) \rightarrow J(P)$ be the map defined by $t_p(I) = I \setminus \lbrace p \rbrace$ if $p \in Z(I)$, $t_p(I) = I \cup \lbrace p \rbrace$ if $p \in U(I)$, and $t_p(I) = I$ otherwise. The following theorem is equivalent to Lemma 1 in Cameron-Fon-Der-Flaass \cite{CameronFonderflaass}. 

\begin{thm}\label{CameronLemma1}
Let $P$ be a poset.  For all linear extensions $(p_1, p_2, \ldots, p_n)$ of $P$ and order ideals $I \in J(P)$, $\Psi(I) = t_{p_1}t_{p_2}\cdots t_{p_n}(I)$.  
\end{thm}

The group $G(P) := \langle t_p \rangle_{p \in P}$ is named the \textit{toggle group} by Striker and Williams \cite{Jessica}.  Note that for all $x$ and $y$, the generators $t_x$ and $t_y$ commute unless $x$ and $y$ share a covering relation.  

In the case that the poset $P$ is ranked, it is natural to consider the linear extensions label the elements of $P$ by order of increasing rank.  For the purposes of this paper, we shall say that $P$ is ranked if there exists an integer-valued function $r$ on $X$ (called the \textit{rank function}) such that $r(p) =0$ for all minimal elements $p \in X$ and, for all $x, y \in X$, if $x$ covers $y$, then $r(x) - r(y) = 1$.  

If $P$ is a ranked poset, let the maximum value of $r$ be $R$.  For all $0 \leq i \leq R$, let $P_i = \lbrace p \in P: r(p) = i\rbrace$, and let $t_i = \prod_{p \in P_i} t_p$.  We see that $t_i$ is always well-defined because, for all $i$, $t_x$ and $t_y$ commute for all $x, y \in P_i$.  By Theorem~\ref{CameronLemma1}, $\Psi = t_0 t_1 \cdots t_R$.  Note that $t_i$ and $t_j$ commute for all $|i-j|>1$.  The following theorem is also a result of Cameron-Fon-Der-Flaass \cite{CameronFonderflaass}. 

\begin{thm}\label{CameronLemma2}
For all permutations $\sigma$ of $\lbrace 0, 1, \ldots, R \rbrace$,  $\Psi_\sigma := t_{\sigma(0)} t_{\sigma(1)} \cdots t_{\sigma(R)}$ is conjugate to $\Psi$ in $G(P)$.
\end{thm}

\begin{cor}
The action $\Psi_\sigma$ has the same orbit structure as $\Psi$ for all $\sigma$.  
\end{cor}

Let $t_{\text{even}} = \prod_{i \text{ even}} t_i$, and let $t_{\text{odd}} = \prod_{i \text{ odd}} t_i$.  It should be clear that $t_{\text{even}}$ and $t_{\text{odd}}$ are well-defined, and it follows from Theorem~\ref{CameronLemma2} that $t_{\text{even}}t_{\text{odd}}$ is conjugate to $\Psi$ in $G(P)$, as noted in the second paragraph of section 4 in \cite{CameronFonderflaass}.  This means that the action of toggling at all the elements of odd rank, followed by toggling at all the elements of even rank, is conjugate to the Fon-Der-Flaass action in the toggle group.  As we shall see, this holds the key to demonstrating that the induced action of every Coxeter element of $W$ on $J(P)$ under $\phi$ is conjugate to the Fon-Der-Flaass action as well.  Striker and Williams made use of the same argument to obtain the conjugacy of promotion and rowmotion (their name for the Fon-Der-Flaass action) in section 6 of \cite{Jessica}, so it should be no surprise that our induced actions reduce to promotion in types $A$ and $B$.  In this sense, our proof of Theorem~\ref{main1} may be considered to be a continuation of their work.  

\section{Minuscule Posets}
In this section, we introduce the primary objects of study for this paper -- the minuscule posets.  We begin with some notation, following Stembridge \cite{Stembridge1}.  Let $\mathfrak{g}$ be a complex simple Lie algebra; let $\mathfrak{h}$ be a Cartan subalgebra; choose a set $\Phi^+$ of positive roots $\alpha$ in $\mathfrak{h}^*$, and let $\Delta = \lbrace \alpha_1, \alpha_2, \ldots, \alpha_n \rbrace$ be the set of simple roots.  Let $(\cdot, \cdot)$ be the inner product on $\mathfrak{h}^*$, and, for each root $\alpha$, let $\alpha^\vee = 2\alpha/(\alpha,\alpha)$ be the corresponding coroot.  Finally, let $\Lambda = \lbrace \lambda \in \mathfrak{h}^* : \alpha \in \Phi \rightarrow (\lambda,\alpha^\vee) \in \mathbb{Z} \rbrace$ be the weight lattice.  

For all $1 \leq i \leq n$, let $s_i$ be the simple reflection corresponding to the simple root $\alpha_i$, and let $W =  \langle s_i \rangle_{1 \leq i \leq n}$ be the Weyl group of $\mathfrak{g}$. If $s$ is conjugate to a simple reflection $s_i$ in $W$, we refer to $s$ as an (abstract) reflection. 

Let $V$ be a finite-dimensional representation of $\mathfrak{g}$.  For each $\lambda \in \Lambda$, let 
$$V_\lambda = \lbrace v \in V : h \in \mathfrak{h} \Longrightarrow hv = \lambda(h)v \rbrace$$
be the weight space corresponding to $\lambda$, and let $\Lambda_V$ be the (finite) set of weights $\lambda$ such that $V_\lambda$ is nonzero.  Recall that there is a standard partial order on $\Lambda$ called the \textit{root order} defined to be the transitive closure of the relations $\mu < \omega$ for all weights $\mu$ and $\omega$ such that $\omega-\mu$ is a simple root.  

\begin{df}\label{weightPoset} \rm
The \textit{weight poset} $Q_V$ of the representation $V$ is the restriction of the root order on $\Lambda$ to $\Lambda_V$.  
\end{df}

If $V$ is irreducible, $Q_V$ has a unique maximal element, which is called the highest weight of $V$.  This leads to the following definition.  

\begin{df}\label{minusculeRep} \rm
Let $V$ be a nontrivial, irreducible, finite-dimensional representation of $\mathfrak{g}$. $V$ is a \textit{minuscule representation} if the action of $W$ on $\Lambda_V$ is transitive.  In this case, the highest weight of $V$ is called the \textit{minuscule weight}.  
\end{df}

\begin{thm}
If $V$ is minuscule, the weight poset $Q_V$ is a distributive lattice.  
\end{thm}

\begin{rem}\rm
This result, due to Proctor (cf. \cite{Proctor}, Propositions 3.2 and 4.1), was originally verified by exhaustive search, but it is also a consequence of Theorem~\ref{minusculelattice}, for which a case-free proof was given using Bruhat-theoretic techniques by Stembridge in \cite{Stembridge2}. 
\end{rem}

\begin{df}\label{minusculedefn} \rm
If $V$ is minuscule, let $P_V$ be the poset of join-irreducible elements of the weight poset $Q_V$, so that $P_V$ is the unique poset satisfying $J(P_V) \cong Q_V$.  Then $P_V$ is the \textit{minuscule poset} of $V$, and posets of this form comprise the \textit{minuscule posets}.  
\end{df}

\begin{rem} \rm
If $V$ is a minuscule representation and $\lambda$ is the highest weight of $V$, we refer to $P_V$ as the \textit{minuscule poset for $\lambda$}.               
\end{rem}
 
\section{Bruhat Posets}

In this section, we develop the framework for the proofs of Theorems~\ref{secondary1} and \ref{secondary2}.  We begin by discussing the Bruhat posets.  Then we establish the connection between these objects and the weight posets of minuscule representations.

We continue with the notation of the previous section.  Given a Weyl group $W$, we define a length function $l$ on the elements of $W$ as follows.  For all $w \in W$, we let $l(w)$ be the minimum length of a word of the form $s_{i_1}s_{i_2} \ldots s_{i_\ell}$ such that $w = s_{i_1}s_{i_2} \ldots s_{i_\ell}$ and $s_{i_j}$ is a simple reflection for all $1 \leq j \leq \ell$.  This allows us to introduce a well-known partial order on $W$, known as the (strong) Bruhat order, for which $l$ also serves as a rank function.  The Bruhat order is defined to be the transitive closure of the relations $w <_B sw$ for all Weyl group elements $w$ and (abstract) reflections $s$ satisfying $l(w) < l(sw)$.  

What is of interest is not the Bruhat order on $W$, but the restrictions of the Bruhat order to parabolic quotients of $W$, for these are the orders that give rise to the Bruhat posets.  

\begin{df} \rm
If $J$ is a subset of $\lbrace 1, 2, \ldots, n \rbrace$, then $W_J := \langle s_i \rangle_{i \in J}$ is the \textit{parabolic subgroup} of $W$ generated by the corresponding simple reflections, and $W^J := W/W_J$ is the \textit{parabolic quotient}.  
\end{df}

It is well-known that each coset in $W^J$ has a unique representative of minimum length, so the quotient $W^J$ may be regarded as the subset of $W$ containing only the minimum-length coset representatives.  This fact facilitates the definition of an analogous partial order on $W^J$.  

\begin{df} \rm
The \textit{Bruhat order} $<_B$ on the parabolic quotient $W^J$ is the restriction of the Bruhat order on $W$ to $W^J$.  Posets of the form $(W^J,<_B)$ comprise the \textit{Bruhat posets}. 
\end{df}

We may also define the left (weak) Bruhat order on $W$ to be the transitive closure of the relations $w <_L sw$ for all Weyl group elements $w$ and \textit{simple} reflections $s$ satisfying $l(w) < l(sw)$.  The analogous partial order on $W^J$ is defined in precisely the same way: $(W^J, <_L)$ is the restriction of $(W, <_L)$ to the minimum-length coset representatives $W^J$.  While the left Bruhat order is not necessary to establish the connection between the minuscule posets and the Bruhat posets, we introduce it here so that our work in this section may be compatible with the theory of fully commutative elements developed in section 5 and exploited in section 6.  

We are now ready to state the following theorem, which appears as Proposition 4.1 in Proctor \cite{Proctor}.  

\begin{thm}\label{ProctorLemma1}
Let $V$ be a minuscule representation with minuscule weight $\lambda$, and let $J = \lbrace i: s_i \lambda = \lambda \rbrace$.  Then $W_J$ is the stabilizer of $\lambda$ in the Weyl group $W$, and the weight poset $Q_V$ is isomorphic to the Bruhat poset $(W^J, <_B)$.   
\end{thm}

\begin{rem}\rm
There is a small subtlety in the proof Theorem~\ref{ProctorLemma1} because the natural map $\varphi: W^J \rightarrow Q_V$ to consider, $w \mapsto w \lambda$, is order-reversing, rather than order-preserving.  (In other words, for all $u, v \in W^J$, $u \lambda < v \lambda$ if and only if $v <_B u$.)  However, composing $\varphi$ with the order-reversing involution of $Q_V$ given by $\omega \mapsto w_0 \omega$, where $w_0$ is the unique longest element of $W$, yields a suitable isomorphism (as noted in the introduction).  Alternatively, $\varphi$ may be precomposed with the corresponding order-reversing involution of $W^J$ given by $w \mapsto w_0 w (w_0^J)^{-1} w_0$, where $w_0^J$ denotes the unique longest element of $W^J$.  In Proctor's proof of Theorem~\ref{ProctorLemma1}, he circumvents this step by defining the partial order on the weights opposite to the root order.  We avoid his approach here because it leads to unnecessary confusion over terms such as ``highest weight.'' 
\end{rem}

\begin{df} \rm
The parabolic quotient $W^J$ is \textit{minuscule} if $W_J$ is the stabilizer of a minuscule weight $\lambda$.  
\end{df}

The assumption that $\mathfrak{g}$ be simple implies that $\lambda$ is fundamental (recall that the fundamental weights $\omega_1, \omega_2, \ldots, \omega_n$ are defined by the condition $(\omega_i, \alpha_j^\vee)= \delta_{ij}$ for all $1 \leq i, j \leq n$, where $\delta_{ij}$ is the Kronecker delta).  Hence if $\lambda = \omega_j$, then $s_i \lambda = \lambda$ for all $i \neq j$.  It follows that if $W^J$ is minuscule, $J = \lbrace 1, 2, \ldots, n \rbrace \setminus \lbrace j \rbrace$, so $W_J$ is a maximal parabolic subgroup of $W$.  In general, a minuscule Bruhat poset is obtained precisely when the ``missing'' element of $J$ is the index of a fundamental weight for which there exists a representation of $\mathfrak{g}$ in which that fundamental weight is minuscule.  

We note that Bruhat posets $W^J$ provide a natural setting for identifying instances of the cyclic sieving phenomenon because they come equipped with a group action, namely that of $W$, and a rank-generating function $W^J(q) := \sum_{w \in W^J} q^{l(w)}$, which is what motivated us to consider them in the first place.  We now turn our attention to the labeling of the minuscule poset $P_V$ and the construction of the isomorphism $\phi: J(P_V) \rightarrow W^J$, which lie behind the proofs of Theorems~\ref{secondary1} and \ref{secondary2}.  

\section{Fully Commutative Elements}

In this section we borrow from Stembridge's theory of fully commutative elements of Weyl groups.  In the next section, we shall see how the theory enables us to characterize the relationship between the action of the Weyl group on the elements of these lattices and the action of the toggle group on the order ideals of the corresponding minuscule posets.

\begin{df} \rm
Let $W$ be a Weyl group, and let $S = \lbrace s_1, s_2, \ldots, s_n \rbrace$ be the set of Coxeter generators.  An element $w \in W$ is \textit{fully commutative} if every reduced word for $w$ can be obtained from every other by means of commuting braid relations only (i.e., via relations of the form $s_js_{j'} = s_{j'}s_j$ for commuting Coxeter generators $s_j$ and $s_{j'}$).  
\end{df}

Given a fully commutative element $w$, we can define a labeled poset $P_w$ that generates all the reduced words of $w$ in the sense that putting labels in the place of poset elements gives a bijection between the linear extensions of $P_w$ and the reduced words of $w$.  

\begin{df}\label{heap} \rm
Let $s_{i_1} s_{i_2} \cdots s_{i_\ell}$ be a reduced word for $w$.  Let $P_w = (\lbrace 1, 2, \ldots, \ell \rbrace, <)$ be a partially ordered set, where the partial order on $\lbrace 1, 2, \ldots, \ell \rbrace$ is defined to be the transitive closure of the relations $j > j'$ for all $j < j'$ in integers such that $s_{i_j}$ and $s_{i_{j'}}$ do not commute.   Then $P_w$ is the \textit{heap} of $w$, and, for all $1 \leq j \leq \ell$, $s_{i_j}$ is the \textit{label} of the heap element $j \in P_w$.  An example is given in Figure~\ref{ExampleA4}.  
\end{df}

\begin{figure}[htp]
\begin{center}
\includegraphics[scale = 0.7]{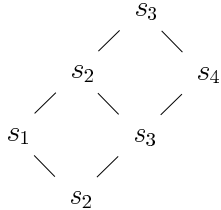} 
\caption{If $W$ is the Weyl group arising from the root system $A_4$, then the element $w := s_3s_2s_4s_1s_3s_2$ is fully commutative, and the heap $P_w$ is as displayed above.  }
\label{ExampleA4}
\end{center}
\vspace{0in}
\end{figure}

Let $\mathcal{L}(P_w):= \lbrace \pi: \pi(1) \geq \pi(2) \geq \ldots \geq \pi(\ell) \rbrace$ be the set of \textit{reverse} linear extensions of $P_w$, and let $\mathcal{L}(P_w,w)$ be the set of labeled reverse linear extensions of $P_w$, i.e., 
$$\mathcal{L}(P_w,w):= \lbrace s_{i_{\pi(1)}} s_{i_ {\pi(2)}} \cdots s_{i_ {\pi(\ell)}} : \pi \in \mathcal{L}(P_w) \rbrace. $$

As alluded to above, the set $\mathcal{L}(P_w,w)$ is significant for the following reason.  

\begin{prop}\label{linearReduced}
$\mathcal{L}(P_w,w)$ is the set of reduced words for $w$ in $W$.  
\end{prop}

\begin{proof}
See Proposition 2.2 in \cite{Stembridge2}.  
\end{proof}

\begin{rem}\rm
We define the partial order on $P_w$ to be the reverse of Stembridge's order and consider reverse linear extensions rather than linear extensions.  Furthermore, Stembridge defines heaps for all words in $W$, whereas our definition is only correct for reduced words of fully commutative elements $w$.  The implications for the theory are rather cosmetic; we make these deviations for the sake of convenience only.
\end{rem}

It follows from Proposition~\ref{linearReduced} that, if $w$ is fully commutative, the heaps of the reduced words for $w$ are all equivalent, so we may refer to the heap of $w$ unambiguously.  This is also noted in \cite{Stembridge2}.   The crucial claim is the next theorem.  

\begin{thm}\label{fullycommutes}
Let $w \in W$ be fully commutative.  Then $J(P_w) \cong \lbrace x \in W: x \leq_L w \rbrace$ is an isomorphism of posets.  
\end{thm}

\begin{proof}
A proof is found in \cite{Stembridge2} (cf. Lemma 3.1), but because our definitions are different from Stembridge's, and because the map between the two posets will be of importance in its own right for our proof of Theorem ~\ref{secondary2}, we provide our own adaptation of Stembridge's proof.

For all $1 \leq k \leq n$, let $C_k := \lbrace j: s_{i_j} = s_k \rbrace$ be the set of all heap elements labeled by $s_k$.  We first note that each $C_k$ is a totally ordered subset of $P_w$.  The proof is by contradiction.  Suppose that there exist incomparable elements $j, j' \in P_w$ such that $s_{i_j}=s_{i_{j'}}=s_k$.  Then there exists a reverse linear extension of $P_w$ in which $j$ and $j'$ occur consecutively, which implies that the corresponding reduced word for $w$ contains two consecutive instances of $s_k$.  This is of course impossible.  Thus, we may write $C_k$ in the form $\lbrace \gamma_{k,1} < \gamma_{k,2} < \ldots < \gamma_{k,{\nu(k,w)}} \rbrace$, where $\nu(k,w)$ denotes the number of instances of $s_k$ in a reduced word for $w$, and $\nu$ is well-defined because $w$ is fully commutative.  

We are now ready to define the bijection between $J(P_w)$ and $\lbrace x \in W : x \leq_L w \rbrace$.  Given an order ideal $I \in J(P_w)$, let $\rho$ be a linear extension of $P_w$ such that $\rho(j) \in I$ for all $1 \leq j \leq |I|$ and $\rho(j) \notin I$ otherwise.   

\begin{df}\label{mapPhi1} \rm
$$
\begin{array}{rcl}
\phi: J(P_w) & \longrightarrow & \lbrace x \in W: x \leq_L w \rbrace \\ 
\end{array}
$$
is defined by
$$
\begin{array}{rcl}
I & \longmapsto & s_{i_{\rho(|I|)}} \cdots s_{i_{\rho(2)}} s_{i_{\rho(1)}}.  \\
\end{array}
$$
\end{df}

\begin{rem} \rm
The choice of the symbol $\phi$ to denote this map is deliberate, for when the heap $P_w$ is minuscule (see Definition~\ref{minusculeheap}), $\phi$ is the map described in the introduction.  
\end{rem}

It is not immediately clear that $\phi$ is well-defined.  However, if $\rho$ and $\rho'$ are both linear extensions of $P_w$ such that $\rho(j), \rho'(j) \in I$ for all $1 \leq j \leq |I|$ and $\rho(j), \rho'(j) \notin I$ otherwise, then let $x = s_{i_{\rho(|I|)}} \cdots s_{i_{\rho(2)}} s_{i_{\rho(1)}}$ and $x' = s_{i_{\rho'(|I|)}} \cdots s_{i_{\rho'(2)}} s_{i_{\rho'(1)}}$.  Since 
$$(\rho(\ell), \ldots, \rho(|I|+1), \rho(|I|), \ldots, \rho(2), \rho(1))$$
is a reverse linear extension of $P_w$, 
$$s_{i_{\rho(\ell)}} \cdots s_{i_{\rho(|I|+1)}} s_{i_{\rho(|I|)}} \cdots s_{i_{\rho(2)}} s_{i_{\rho(1)}}$$
is a reduced word for $w$, so $s_{i_{\rho(\ell)}} \cdots s_{i_{\rho(|I|+1)}}$ is a reduced word for $wx^{-1}$.  However, 
$$(\rho(\ell), \ldots, \rho(|I|+1), \rho'(|I|), \ldots, \rho'(2), \rho'(1))$$
is also a reverse linear extension of $P_w$.  It follows that $s_{i_{\rho(\ell)}} \cdots s_{i_{\rho(|I|+1)}}$ is a reduced word for $wx'^{-1}$, so $x=x'$, as desired.  

To see that $\phi$ is bijective, we define the inverse map $\phi^{-1}: \lbrace x \in W: x \leq_L w \rbrace \rightarrow J(P_w)$ by $x \mapsto \cup_{k=1}^{n} \lbrace \gamma_{k,h}: 1 \leq h \leq \nu(k,x) \rbrace$.  Because every reduced word for $x$ is the final segment of a reduced word for $w$, it should be clear that $\phi^{-1}(x)$ is an order ideal of $P_w$ for all $x \leq_L w$.  It is a trivial matter to verify that $\phi^{-1} \phi$ is the identity on $J(P_w)$ and $\phi \phi^{-1}$ is the identity on $\lbrace x \in W: x \leq_L w \rbrace$, so this completes the proof.  
\end{proof}

The following theorem demonstrates the relevance of the theory of fully commutative elements to our main results.  

\begin{thm}\label{minusculelattice}
If $W^J$ is minuscule, then the following three claims hold:

\begin{enumerate}[label=(\roman{*})]
\item If $w \in W^J$, $w$ is fully commutative;
\item $(W^J, <_L)$ is a distributive lattice;
\item $(W^J, <_B) = (W^J, <_L)$.    
\end{enumerate}
\end{thm}

This theorem is a consequence of Theorems 6.1 and 7.1 in \cite{Stembridge2}, for which Stembridge's proofs are uniform.  We will see how it enables us to apply our knowledge of fully commutative heaps to the minuscule setting in the next section.

\section{The Main Results}

In this section, we prove Theorems~\ref{secondary1} and \ref{secondary2}.  We start with the following definition and subsequent theorem.  

\begin{df}\label{minusculeheap} \rm
If $W^J$ is minuscule, and $w_0^J$ is the longest element of $W^J$, then the heap $P_{w_0^J}$ is \textit{minuscule}, and heaps of this form comprise the \textit{minuscule heaps}.  
\end{df}

\begin{rem} \rm
Some of the minuscule heaps appear in Wildberger \cite{Wildberger}, but his construction differs from ours.  In particular, he introduces a set of heaps that he calls \textit{two-neighbourly}, and he observes that these are precisely the minuscule heaps arising from complex simple Lie algebras whose root systems are simply laced.  
\end{rem}

\begin{thm}\label{mainTheorem}
Let $V$ be a minuscule representation of a complex simple Lie algebra $\mathfrak{g}$ with minuscule weight $\lambda$ and Weyl group $W$.  If $S = \lbrace s_1, s_2, \ldots, s_n \rbrace$ is the set of Coxeter generators and $W_J$ is the maximal parabolic subgroup stabilizing $\lambda$, then the following claims hold:

\begin{enumerate}[label=(\roman{*})]
\item If $w_0^J$ is the longest element of $W^J$, then the poset $\lbrace x \in W: x \leq_L w_0^J \rbrace$ and the lattice $(W^J, <_L)$ are identical, and, furthermore, the minuscule heap $P_{w_0^J}$ and the minuscule poset $P_V$ are isomorphic as posets.  
\item The isomorphism $\phi: J(P_{w_0^J}) \rightarrow \lbrace x \in W: x \leq_L w_0^J \rbrace \cong (W^J, <_L) \cong (W^J, <_B)$ defined in Definition~\ref{mapPhi1} satisfies the following property: For all $1 \leq k \leq n$, the induced action of the Coxeter generator $s_k$ on $J(P_{w_0^J})$ in the toggle group $G(P_{w_0^J})$ may be expressed in the form $\displaystyle\prod_{\substack{p \in P_{w_0^J} \\ p \text{ is labeled by }s_k}} t_p$.
\end{enumerate}

\end{thm}

\begin{exam} \rm
In the case when the root system is $A_4$ and the minuscule weight is $\omega_2$, Figure~\ref{fig: mainTheorem} shows the minuscule heap $P_{s_3s_2s_4s_1s_3s_2}$ (on the left) and the corresponding Bruhat poset $(W^J, <_B)$ (on the right).  If $I$ is the order ideal encircled by the solid line, then $\phi(I)$ is the coset representative encircled by the solid line, and $\prod_{p \in P_{s_3s_2s_4s_1s_3s_2} \text{ is labeled by }s_2} t_p(I)$ is the order ideal encircled by the dotted line.  Furthermore, $\phi( \prod_{p \in P_{s_3s_2s_4s_1s_3s_2} \text{ is labeled by }s_2} t_p(I)) = s_2 \phi(I)$ is the coset representative encircled by the dotted line, thus illustrating the statement \textit{(ii)} in Theorem~\ref{mainTheorem}.    
\end{exam}

\begin{figure}[htp]
\begin{center}
\subfigure[]{\includegraphics[scale = 0.7]{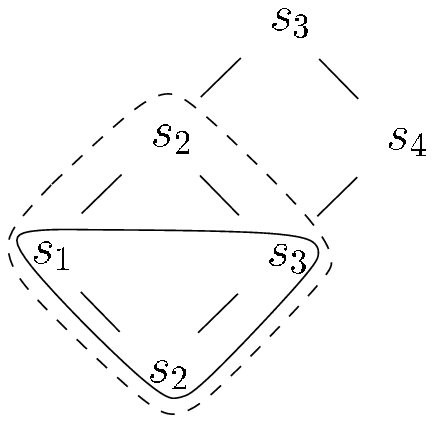}} \hspace{0.7in}
\subfigure[]{\includegraphics[scale = 0.7]{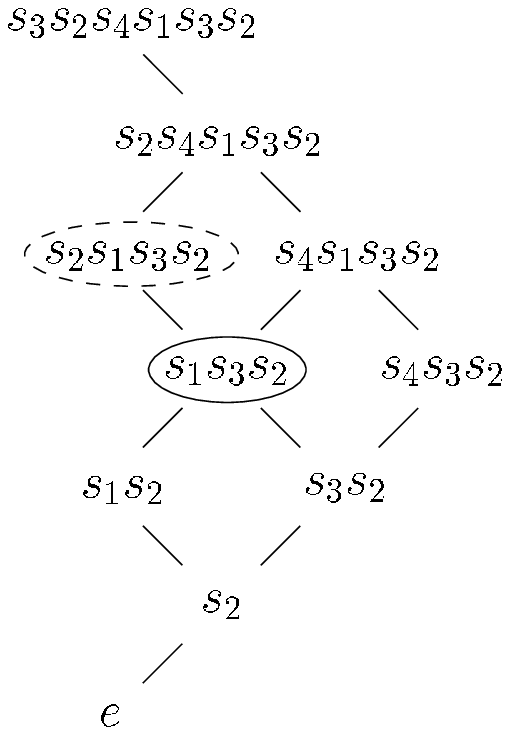}}
\caption{The map $\phi$ sends the indicated order ideals to the indicated coset representatives. }
\label{fig: mainTheorem}
\end{center}
\vspace{0in}
\end{figure}

\begin{proof}

\noindent \textit{(i)} By Proposition 2.6 in \cite{Stembridge2}, $w_0^J$ is the unique maximal element of $(W^J, <_L)$.  It follows that if $x \in W^J$, $x \leq_L w_0^J$.  To see that the converse also holds, let $x_0$ be the longest element of $W_J$, and note that $x \in W^J$ if and only if $xx_0$ is reduced (i.e. if and only if the product of a reduced word for $x$ and a reduced word for $x_0$ is necessarily a reduced word for $xx_0$).  If $x \leq_L w_0^J$, then there exists a reduced word for $x$ that is the final segment of a reduced word for $w_0^J$.  Since $w_0^J x_0$ is reduced, there exists a reduced word for $x$ and a reduced word for $x_0$ such that their product is a reduced word for $xx_0$, and it follows from the fact that all reduced words for the same element are of the same length that $xx_0$ is reduced.  We may conclude that $x \in W^J$, so, in general, $\lbrace x \in W: x \leq_L w_0^J \rbrace = (W^J, <_L)$.  However, $J(P_{w_0^J}) \cong \lbrace x \in W: x \leq_L w_0^J \rbrace$, and $(W^J, <_L) = (W^J, <_B) \cong J(P_V)$ by Definition~\ref{minusculedefn} and Theorems~\ref{ProctorLemma1} and \ref{minusculelattice}, so $J(P_{w_0^J}) \cong J(P_V)$, and it follows that $P_{w_0^J} \cong P_V$ is an isomorphism of posets, as desired.   \\ 

\noindent \textit{(ii)} Because $(W^J, <_L) = (W^J, <_B)$, it suffices to prove the claim with $(W^J, <_L)$ in place of \\
$(W^J, <_B)$.  Following the notation in the proof of Theorem~\ref{fullycommutes}, for all $1 \leq k \leq n$, let $C_k$ be the set of all heap elements labeled by $s_k$, and let $t'_k$ be the toggle group element defined by $t'_k = \prod_{p \in C_k} t_p$.  From section 5, we know that $C_k$ is totally ordered, and, by definition of $P_{w_0^J}$, no two elements of $C_k$ share a covering relation, so it follows that $t'_k$ is well-defined for all $k$.  Now let $I$ be an order ideal in $J(P_{w_0^J})$, let $w = \phi(I)$, and let $\ell$ denote the length of the longest coset representative $w_0^J$.  Consider the following lemmas: 

\begin{lem}\label{shortlemma}
The order ideal $t'_k(I)$ disagrees with $I$ on at most one vertex of $P_{w_0^J}$.  
\end{lem}
\begin{proof}
It suffices to show that if there exists one vertex on which the two disagree, then there cannot exist any other such vertices.  If there exists a vertex on which the two disagree, then there must exist a vertex $p_0$ labeled by $s_k$ such that $p_0 \in Z(I)$ or $p_0 \in U(I)$.  Without loss of generality, let $p_0 \in Z(I)$, and assume that the toggles $t_p$ are applied to $I$ in order of increasing $p$.  Then for all $p \neq p_0$, the toggle at $p$ has no effect, for $p_0$ is in the order ideal when $t_p$ is applied if and only if $p< p_0$.  
\end{proof}

\begin{lem}\label{apple}
There exists an element $p_0 \in P$ such that $p_0 \in Z(I)$ if and only if $s_kw$ is not reduced.  In this case, if $s_{i_{l(s_kw)}} \cdots s_{i_2} s_{i_1}$ is a reduced word for $s_kw$, then $s_k s_{i_{l(s_kw)}} \cdots s_{i_2} s_{i_1}$ is a reduced word for $w$, and $\phi(I \setminus \lbrace p_0 \rbrace) = s_{i_{l(s_kw)}} \cdots s_{i_2} s_{i_1}$.  
\end{lem}
\begin{proof}
If $p_0 \in Z(I)$, let $s_{i_\ell} \cdots s_{i_2} s_{i_1}$ be a reduced word for $w_0^J$, and assume that the heap $P_{w_0^J}$ is built with reference to this particular reduced word (recall that the heaps of every reduced word for $w_0^J$ are equivalent).  Since $p_0 \in Z(I)$, $I \setminus \lbrace p_0 \rbrace$ is an order ideal of $P_{w_0^J}$.  Let $(\rho(1), \rho(2), \ldots, \rho(|I|-1))$ be a linear extension of $I \setminus \lbrace p_0 \rbrace$ (i.e. a linear extension of the poset with vertices in $I \setminus \lbrace p_0 \rbrace$ and partial order given by the restriction of the partial order on $P_{w_0^J}$ to $I \setminus \lbrace p_0 \rbrace$).  Then $(\rho(1), \rho(2), \ldots, \rho(|I|-1), p_0)$ is a linear extension of $I$.  Let $\rho(|I|) = p_0$, and extend this linear extension to a linear extension of $P_{w_0^J}$, $(\rho(1), \rho(2), \ldots, \rho(\ell))$.  By definition of $\phi$, $s_{i_{\rho(|I|)}} s_{i_{\rho(|I|-1)}} \cdots s_{i_{\rho(2)}} s_{i_{\rho(1)}}$ is a reduced word for $w$.  Since $p_0$ is labeled by $s_k$, $s_{i_{p_0}} = s_k$, so it follows that $s_kw = s_{i_{\rho(|I|-1)}} \cdots s_{i_{\rho(2)}} s_{i_{\rho(1)}}$.  This implies that $s_kw$ is not reduced.  

If $s_kw$ is not reduced, let $s_{i_{l(s_kw)}} \cdots s_{i_2} s_{i_1}$ be a reduced word for $s_kw$.  Note that $s_k s_{i_{l(s_kw)}} \cdots s_{i_2} s_{i_1}$ is a reduced word for $w$, else $l(w) = l(s_k s_{i_{l(s_kw)}} \cdots s_{i_2} s_{i_1}) < l(s_{i_{l(s_kw)}} \cdots s_{i_2} s_{i_1}) = l(s_kw) < l(w)$, which is absurd.  Let $i_{l(w)}=k$, and extend this word to a reduced word for $w_0^J$, $s_{i_\ell} \cdots s_{i_2} s_{i_1}$.  Without loss of generality, we may assume that the heap $P_{w_0^J}$ is built with reference to this particular reduced word, in which case the vertex corresponding to $s_{i_{l(w)}}$ is maximal over the vertices in the order ideal $\phi^{-1}(w)=I$.  This implies that there exists an element of $P_{w_0^J}$ labeled by $s_k$ that belongs to $Z(I)$ and $\phi(I \setminus \lbrace p_0 \rbrace) = s_{i_{l(s_kw)}} \cdots s_{i_2} s_{i_1}$.    
\end{proof}

\begin{lem}\label{pear}
There exists an element $p_0 \in P$ such that $p_0 \in U(I)$ if and only if $s_kw$ is reduced and $s_kw \in W^J$.  In this case, if $s_{i_{l(w)}} \cdots s_{i_2} s_{i_1}$ is a reduced word for $w$, then $s_k s_{i_{l(w)}} \cdots s_{i_2} s_{i_1}$ is a reduced word for $s_kw$, and $\phi(I \cup \lbrace p_0 \rbrace) = s_k s_{i_{l(w)}} \cdots s_{i_2} s_{i_1}$.  
\end{lem}
\begin{proof}
This result is analogous to Lemma~\ref{apple}.    
\end{proof} 

Now we take up the proof of Theorem~\ref{mainTheorem}.  

\begin{itemize}

\item If $s_kw$ is not reduced, then, by Lemma~\ref{apple}, there exists an element $p_0 \in P$ labeled by $s_k$ such that $p_0 \in Z(I)$, so, by Lemma~\ref{shortlemma}, $\prod_{p \in C_k} t_p(I) = I \setminus \lbrace p_0 \rbrace$.  If $s_{i_{l(s_kw)}} \cdots s_{i_2} s_{i_1}$ is a reduced word for $s_kw$, it follows from Lemma~\ref{apple} that $s_k s_{i_{l(s_kw)}} \cdots s_{i_2} s_{i_1}$ is a reduced word for $w$ and $\phi(I \setminus \lbrace p_0 \rbrace) = s_{i_{l(s_kw)}} \cdots s_{i_2} s_{i_1}$, as desired.  
\item If $s_kw$ is reduced and $s_kw \in W^J$, then, by Lemma~\ref{pear}, there exists an element $p_0 \in P$ labeled by $s_k$ such that $p_0 \in U(I)$, so, by Lemma~\ref{shortlemma}, $\prod_{p \in C_k} t_p(I) = I \cup \lbrace p_0 \rbrace$.  If $s_{i_{l(w)}} \cdots s_{i_2} s_{i_1}$ is a reduced word for $w$, it follows from Lemma~\ref{pear} that $s_k s_{i_{l(w)}} \cdots s_{i_2} s_{i_1}$ is a reduced word for $s_kw$ and $\phi(I \cup \lbrace p_0 \rbrace) = s_k s_{i_{l(w)}} \cdots s_{i_2} s_{i_1}$, as desired.  
\item If $s_kw$ is reduced and $s_kw \notin W^J$, it follows from Lemmas~\ref{apple} and \ref{pear} that no elements of $P_{w_0^J}$ labeled by $s_k$ belong to $Z(I)$ or $U(I)$, so, by Lemma~\ref{shortlemma}, $\prod_{p \in C_k} t_p(I) = I$.  However,  $s_kw$ covers $w$ in the left Bruhat order on $W$, so, by Corollary 2.5.2 in Bj\"orner-Brenti \cite{Bjorner}, it follows that $s_kw=ws_j$, where $j \in J$.  Since $s_j \in W_J$, we may conclude that $s_kwW_J = ws_jW_J = wW_J$, as desired.  

\end{itemize}

\end{proof}

\begin{rem}\rm
In the proofs of Lemmas~\ref{shortlemma} and \ref{pear}, we suppressed the cases in which elements of $P_{w_0^J}$ were given or shown to belong to $U(I)$ rather than $Z(I)$ because the conditions are symmetric, so the arguments are identical.  However, we originally wrote out proofs that address both cases independently, and these may be found in the REU report \cite{RushShi}.    
\end{rem}

We proceed to the proofs of Theorems~\ref{secondary1} and \ref{secondary2}.  

\subsection{Proof of Theorem~\ref{secondary2}}

Let $P_V$ be a minuscule poset, and label each element of $P_V$ by the label of the corresponding element of $P_{w_0^J}$.  From Theorem~\ref{mainTheorem}, it follows that, for all $1 \leq k \leq n$, the following diagram is commutative:

$$
\begin{array}{rcl}
J(P_V) &\overset{\phi}{\longrightarrow}& W^J\\
t'_k \downarrow& & s_k \downarrow \\
J(P_V) &\overset{\phi}{\longrightarrow}& W^J.\\
\end{array}
$$

For any ordering of $S = (s_{i_1}, s_{i_2}, \ldots, s_{i_n})$, $c = s_{i_1} s_{i_2} \cdots s_{i_n}$ and $t_{(i_1,i_2, \ldots, i_n)} = t'_{i_1} t'_{i_2} \cdots t'_{i_n}$, so Theorem~\ref{secondary2} follows immediately.  \qed

\subsection{Proof of Theorem~\ref{secondary1}}

Let $P_V$ be a minuscule poset, and again label each element of $P_V$ by the label of the corresponding element of $P_{w_0^J}$.  Theorem~\ref{mainTheorem} embeds the Weyl group $W$ as a subgroup of the toggle group $G(P_V)$.  In light of Theorem~\ref{secondary2}, since the Coxeter elements are known to be pairwise conjugate in $W$, it suffices to exhibit a particular ordering $S= (s_{i_1}, s_{i_2}, \ldots, s_{i_n})$ such that $t_{(i_1,i_2, \ldots, i_n)} = t'_{i_1} t'_{i_2} \cdots t'_{i_n}$ is conjugate to $\Psi$ in $G(P_V)$.  However, in section 2, we saw that $t_{\text{even}} t_{\text{odd}}$ is conjugate to $\Psi$ in $G(P_V)$.  What we prove here is that there exists an ordering $(s_{i_1}, s_{i_2}, \ldots, s_{i_n})$ such that the toggle group elements $t'_{i_1} t'_{i_2} \cdots t'_{i_n}$ and $t_{\text{even}} t_{\text{odd}}$ are equal.  

We start with two lemmas:

\begin{lem}\label{minusculeRanked}
If $P$ is a minuscule poset, then $P$ is a ranked poset.  
\end{lem}
\begin{proof}
As discussed in the introduction, minuscule posets are Gaussian (cf. Proctor \cite{Proctor}, Theorem 6).  The fact that all Gaussian posets are ranked is recorded as Exercise 25(b) in Stanley \cite{Stanleyec}.  
\end{proof}

\begin{lem}\label{bipartiteDiagram}
If $W$ is the Weyl group of a complex simple Lie algebra $\mathfrak{g}$, then the Dynkin diagram of the associated root system is acyclic and therefore bipartite.  
\end{lem}
\begin{proof}
See Chapter 1, Exercise 4 in Bj\"orner-Brenti \cite{Bjorner}.  
\end{proof}

Let $r$ be the rank function for $P_V$.  For all $1 \leq k \leq n$, we claim that the ranks of all the vertices labeled by $s_k$ are of the same parity.  

For the proof, the key observation is that each covering relation in the heap $P_{w_0^J}$ corresponds to an edge of the Dynkin diagram of the associated root system.  (Recall that the partial order on $P_{w_0^J}$ is the transitive closure of the relations $j > j'$ for all $j < j'$ in integers such that $s_{i_j}$ and $s_{i_{j'}}$ do not commute; cf. Definition~\ref{heap}.)  Assume for the sake of contradiction that there exists a $k$ such that $j, j' \in P_{w_0^J}$ are both labeled by $s_k$ but $r(j')-r(j)$ is odd.  Without loss of generality, let $r(j')>r(j)$.  Since $C_k$ is totally ordered, it follows that $j'>j$, and that there exists a set of vertices $\lbrace j_1, j_2, \ldots, j_{2u} \rbrace$ such that $j'$ covers $j_1$, $j_{2u}$ covers $j$, and $j_i$ covers $j_{i+1}$ for all $1 \leq i \leq 2u-1$.  We may conclude that there exists a path of odd length in the Dynkin diagram from the vertex corresponding to the $k^{\text{th}}$ simple root to itself.  However, by Lemma~\ref{bipartiteDiagram}, the graph of the Dynkin diagram is bipartite, so this is impossible.  

Note that if $p, p' \in P_V$ and $r(p) \equiv r(p') \pmod{2}$, then $t_p$ and $t_{p'}$ commute in $G(P_V)$.  Let $S_{\text{odd}}$ be the set of all $k$ such that $s_k$ is a simple reflection and the rank of $p$ is odd for all vertices $p \in P_{w_0^J}$ labeled by $s_k$.  Similarly, let $S_{\text{even}}$ be the set of all $k'$ such that $s_{k'}$ is a simple reflection and the rank of $p$ is even for all vertices $p \in P_{w_0^J}$ labeled by $s_{k'}$.  It follows that $t_{\text{even}} t_{\text{odd}} = \prod_{k' \in S_{\text{even}}} t'_{k'} \prod_{k \in S_{\text{odd}}} t'_k$.  This completes the proof.  \qed  

\section{Plane Partitions, Preliminaries}

For the remainder of the paper, we shift the focus from the representation-theoretic aspects of the minuscule posets to their combinatorial properties.  We begin by recalling the Gaussian criterion from the introduction.  Since all Gaussian posets are ranked (as noted in the proof of Lemma~\ref{minusculeRanked}), we state it for ranked posets.

\begin{df}\label{Gaussian} \rm
Let $P$ be a ranked poset with rank function $r$.  $P$ is \textit{Gaussian} if, for all positive integers $m$, the following equality holds:
$$J(P \times [m];q) = \prod_{p \in P} \frac{1 - q^{m+r(p)+1}}{1 - q^{r(p)+1}}.$$

\end{df}  

\begin{rem}\rm
Not only are all minuscule posets Gaussian, but, interestingly enough, it is conjectured that all Gaussian posets are minuscule.
\end{rem}

The two most important known families of Gaussian posets are the first two infinite families of minuscule posets.  The order ideals of Cartesian products of these posets with chains may be identified with the combinatorial objects that we refer to as plane partitions.  Thus, in establishing Theorem~\ref{main2} for these cases, we are implicitly formulating new combinatorial identities for plane partitions that come already organized and explained.  

\begin{df} \label{planePartition} \rm
A \textit{plane partition} of an integer $x$ is a two-dimensional array of nonnegative integers $\lbrace x_{i, j} \rbrace_{i, j \geq 1}$ satisfying $x = \sum_{i, j \geq 1} x_{i,j}$ and $x_{i,j} \geq  x_{i, j+1}, x_{i+1, j} $ for all $i, j \geq 1$. 
\end{df}

We say that a plane partition $\lbrace x_{i, j} \rbrace$ is inside $k \times n \times m$ if $0 \leq x_{i, j} \leq m$ for all $1 \leq i \leq k$, $1 \leq j \leq n$, and $x_{i, j} = 0$ otherwise.  If we think of a plane partition as a polyhedron composed of stacks of unit cubes, with $x_{i,j}$ cubes stacked on the $(i, j)^\text{th}$ square for all $(i,j)$, then it should be clear that any such plane partition corresponds to an order ideal of $[k] \times [n] \times [m]$.  Conversely, any order ideal of $[k] \times [n] \times [m]$ corresponds to a plane partition inside $k \times n \times m$.  In other words, there is a canonical bijection between the plane partitions inside $k \times n \times m$ and the order ideals of $[k] \times [n] \times [m]$. 

\begin{thm}\label{MacMahon}
(MacMahon) The following generating function counts plane partitions $\pi$ inside $k \times n \times m$ by their cardinality: 
$$\sum_{\pi \subset k \times n \times m} q^{|\pi|} = \prod_{\substack{1 \leq i \leq k \\ 1 \leq j \leq n \\1\leq l \leq m}} \frac{[i + j + l - 1]_q}{[i + j + l - 2]_q}. $$
\end{thm}

\begin{rem}\rm
It follows that the MacMahon formula is also the rank-generating function for the poset of order ideals of $[k] \times [n] \times [m]$.  Hence Theorem~\ref{MacMahon} is equivalent to the claim that the minuscule poset $[k] \times [n]$ is Gaussian.  
\end{rem} 

\begin{df} \rm
A \textit{symmetric plane partition} of an integer $x$ is a plane partition as defined in Definition~\ref{planePartition}, subject to the additional condition $x_{i, j} = x_{j, i}$ for all $i, j \geq 1$.  
\end{df}

As expected, there is a canonical bijection between the symmetric plane partitions inside $n \times n \times m$ and the order ideals of the poset $([n] \times [n])/S_2 \times [m]$.  The analogue to the MacMahon formula for symmetric plane partitions is called the Bender-Knuth formula (cf. \cite{Andrews}).  

\begin{thm}\label{benderKnuth} 
The generating function that counts symmetric plane partitions $\pi$ inside $n \times n \times m$ by their cardinality is as follows:
$$\sum_{\substack{{\pi \subset n \times n \times m} \\ { \pi \text{ symmetric}}}  } q^{|\pi|} = \prod_{i = 1}^n \left(\frac{[m+2i-1]_q}{[2i-1]_q} \prod_{h = i+1}^n \frac{[2(m+i+h-1)]_q}{[2(i+h-1)]_q}  \right). $$
\end{thm}

\begin{rem}\rm
By similar reasoning, the Bender-Knuth formula is the rank-generating function for the poset of order ideals of $([n] \times [n])/S_2 \times [m]$.  Hence Theorem~\ref{benderKnuth} is equivalent to the claim that the minuscule poset $([n] \times [n])/S_2$ is Gaussian.  
\end{rem} 

\section{Proof of Theorem~\ref{main2} for the First Infinite Family}\label{mn2}


\subsection{Proof of Theorem~\ref{main2}}

In the REU report \cite{RushShi}, we presented a proof of Theorem~\ref{main2} for the first infinite family by direct analysis of a collection of \textit{bracket sequences} introduced by Cameron and Fon-Der-Flaass, which are shown in \cite{CameronFonderflaass} to be in bijection with the order ideals of $[k] \times [n] \times [2]$.  Inspired by our work, Striker and Williams obtained a bijection in \cite{Jessica} between these bracket sequences and non-crossing partitions of $\lbrace 1, 2, \ldots, k+n+1 \rbrace$ into $k+1$ parts.  In Theorem 7.2 of Reiner-Stanton-White \cite{Reiner}, it is proven that the latter exhibits the cyclic sieving phenomenon with respect to the $q$-analogue of the Narayana number, which is the same as the $q$-analogue of the MacMahon formula.  It follows that Theorem~\ref{main2} holds in this case.  \qed 

In light of the Striker-Williams simplifications, we omit our proof in this article.  However, we remain hopeful that the techniques we used to analyze the bracket sequences can lead to a direct approach to establishing the cyclic sieving phenomenon in more complicated posets, e.g., $([n] \times [n])/S_2 \times [3]$, for which we conjecture that the cyclic sieving phenomenon holds.  The original proof may still be found online in \cite{RushShi}.


\subsection{Substituting Roots of Unity into the MacMahon Formula} \label{plugRoots1}

Knowing that the desired cyclic sieving phenomenon holds in the case $[k] \times [n] \times [2]$, we investigate some of its predictions.  In \cite{RushShi}, this data was used to verify that the cyclic sieving phenomenon holds for the first infinite family, but here its primary purpose is actually to help verify that the cyclic sieving phenomenon holds for the second infinite family, which will be shown in section~\ref{second}.  

Letting $m = 2$ in Theorem~\ref{MacMahon}, we see that the rank-generating function for $J([k] \times [n] \times [2])$ becomes 
$$ J([k] \times [n] \times [2]; q) =  \begin{bmatrix} k + n + 1 \\ n \end{bmatrix}_q \begin{bmatrix} k + n \\ n \end{bmatrix}_q\frac{[1]_q}{[n+1]_q}. $$

\begin{lem}\label{rootEval}
Let $n = n'd + r$ and $k = k'd + s$, where $0 \leq r, s \leq d-1$. Then 
$$\begin{bmatrix}n \\ k \end{bmatrix}_{q = e^{\frac{2 \pi i}{d}}} = \binom{n'}{k'} \begin{bmatrix}r \\ s \end{bmatrix}_{q = e^{\frac{2 \pi i }{d}}}$$
\end{lem}
\begin{proof} 
See Proposition 2.1 in Guo-Zeng \cite{GuoZeng}.  
\end{proof}

\begin{prop} \label{ddivides}
If $\ell$ is a proper divisor of $k + n + 1$, let $d = \frac{k + n + 1}{\ell}$.  Then there are no orbits of order $\ell$ unless $d|n$ or $d|n+1$.   
\end{prop}

\begin{proof}

Let $q := \left(e^{\frac{2 \pi i}{m+n+1}}\right)^\ell$ be a primitive $d^{\text{th}}$ root of unity.  Expanding the MacMahon formula, we have the following: 
$$\begin{bmatrix} k + n + 1 \\ n \end{bmatrix}_q \begin{bmatrix} k + n \\ n \end{bmatrix}_q  \frac{[1]_q}{[n+1]_q} = \left(\frac{[k+n+1]_q \cdots [k+2]_q}{[n]_q \cdots [1]_q} \right) \left(\frac{[k+n]_q \cdots [k+1]_q}{[n]_q \cdots [1]_q} \right) \frac{[1]_q}{[n+1]_q}. $$
Suppose $d \nmid n$ and $d \nmid n+1$; then, since $d | k + n + 1$, by Lemma ~\ref{rootEval}, it follows that 
$$\begin{bmatrix} k+n+1 \\ n \end{bmatrix}_{q = e^{\frac{2 \pi i }{d}}} = \binom{\ell}{\ell'} \begin{bmatrix} r \\ r' \end{bmatrix}_{q = e^{\frac{ 2 \pi i }{d}}}, $$
where $k + n + 1 = \ell d + r$, $n = \ell' d + r'$, and $0 \leq r, r' \leq d-1$.  Since $d | k + n + 1$ and $d \nmid n$, it should be clear that $r = 0$ and $r' > 0$.  Therefore, the expression evaluates to 0, as desired. 
\end{proof} 

\subsection{What happens in $[k] \times [n] \times [m]$ when $m \geq 3$?  }
It has been verified via Dilks's Maple code that cyclic sieving does not occur in the poset $[3] \times [3] \times [3]$.  Furthermore, the order of the Fon-Der-Flaass action is 33 for the poset $[4] \times [4] \times [4]$, so, in particular, it is not true in general that the order of the Fon-Der-Flaass action is $k+n+m-1$ for the poset $[k] \times [n] \times [m]$.  However, it is conjectured by Cameron and Fon-Der-Flaass that if $k+n+m-1$ is prime, then the order of the Fon-Der-Flaass action is divisible by $k+n+m-1$, and they have proved this in \cite{CameronFonderflaass} for all posets in which $m$ exceeds $(k-1)(n-1)$.  

\section{Proof of Theorem~\ref{main2} for the Second Infinite Family} \label{second}

In this section, we obtain the fact that the cyclic sieving phenomenon holds for all posets of the form $([n] \times [n])/S_2 \times [2]$ as a consequence of the fact that the cyclic sieving phenomenon holds for all posets of the form $[n] \times [n] \times [2]$.

Note that the Fon-Der-Flaass action on ordinary plane partitions viewed as order ideals of $[n] \times [n] \times [2]$ restricts to the Fon-Der-Flaass action on symmetric plane partitions viewed as order ideals of $([n] \times [n])/S_2 \times [2]$.  Hence the order of $\Psi$ on order ideals of $([n] \times [n])/S_2 \times [2]$ divides $2n+1$.  Since the orbit of the empty order ideal is of length $2n+1$, it follows that the order of $\Psi$ is in fact equal to $2n+1$.  

From the Bender-Knuth formula (Theorem~\ref{benderKnuth}), we see that the rank-generating function for $J(([n] \times [n])/S_2 \times [2])$ is $\begin{bmatrix} 2n+1 \\ n \end{bmatrix}_q$.  It follows by Lemma~\ref{rootEval} that if $q$ is a $(2n+1)^{\text{th}}$ root of unity, substituting $q$ into the Bender-Knuth expression yields 0 unless $q = 1$.  Since we know that substituting $q=1$ gives the total number of order ideals of $([n] \times [n])/S_2 \times [2]$, it suffices to show that all the orbits of order ideals of $([n] \times [n])/S_2 \times [2]$ are free orbits of length $2n+1$. 

Assume for the sake of contradiction that there exists an orbit of length $\frac{2n+1}{d}$, where $d > 1$.  As discussed, this orbit must also arise in the case $[n] \times [n] \times [2]$.  However, it follows from Proposition~\ref{ddivides} that $d|n$ or $d| n+1$.  Since $\gcd(n, 2n+1) = \gcd(n+1,2n+1) =1$, this contradicts the assumption $d>1$, so we may conclude that all the orbits are free orbits of length $2n+1$, as desired.  \qed

\subsection{What happens in $([n] \times [n])/S_2 \times [m]$ when $m \geq 3$?  }
It has been verified via Dilks's Maple code that cyclic sieving does not occur in the poset $([6] \times [6])/S_2 \times [4]$.  However, every poset of the form $([n] \times [n])/S_2 \times [3]$ that we tested was found to obey the cyclic sieving phenomenon, so it is tempting to conjecture that the cyclic sieving phenomenon holds for all such posets.    

\section{Proof of Theorem~\ref{main2} for the Third Infinite Family}

Remarkably, the following theorem is true.  

\begin{thm}
For all positive integers $r$, if $P := J^{n-3}([2] \times [2])$ is a minuscule poset belonging to the third infinite family, then the triple $(J(P \times [m]), J(P \times [m];q), \Psi)$ exhibits the cyclic sieving phenomenon for all positive integers $m$.  
\end{thm}
The proof of this result may be found in the REU report \cite{RushShi}.  It is accomplished by a bijection between order ideals of $J^{n-3}([2] \times [2]) \times [m]$ and the same Cameron Fon-Der-Flaass bracket sequences that arise in the case $[k] \times [n] \times [2]$.  To manipulate these bracket sequences, we devised a rule analogous to that in Cameron-Fon-Der-Flaass \cite{CameronFonderflaass} that differs in the details.   

\section{Proof of Theorem~\ref{main2} for the Exceptional Cases}

We verified via Dilks's code that, if $P$ is the first exceptional poset, the cyclic sieving phenomenon holds for the triple $(J(P \times [m]), J(P \times [m];q), \Psi)$ when $1 \leq m \leq 4$, and, if $P$ is the second exceptional poset, the cyclic sieving phenomenon holds for the triple $(J(P \times [m]), J(P \times [m];q), \Psi)$ when $1 \leq m \leq 3$.  For reference, in Table~\ref{exceptionalCases} we provide the data on the orbit structures corresponding to both exceptional posets for the cases $m=1$ and $m=2$.  This data is not required to establish Theorem~\ref{main1} because the proof of Theorem~\ref{mainTheorem} is uniform, but it is required to show that Theorem~\ref{main2} holds in the exceptional cases and not just in the infinite families.  

\begin{center}
\begin{table}[htp]
\begin{tabular}{|c|c|c|}
\hline
 & $P = J^2([2] \times [3])$ & $P = J^3([2] \times [3])$ \\  \hline
$m=1$ & $2 \times 12 + 1 \times 3$ & $3 \times 18 + 1 \times 2$ \\ \hline
 $m=2$ & $27 \times 13$ & $77 \times 19$ \\ \hline
\end{tabular}
\caption{In each entry, the table displays the number of Fon-Der-Flaass action orbits of each size that occur in the indicated poset of order ideals.  For instance, the poset of order ideals $J(J^3([2] \times [3]) \times [1])$ is composed of 3 orbits of order 18 and 1 orbit of order 2.}
\label{exceptionalCases}
\end{table}
\end{center}

\vspace{-0.3in}

\subsection{What happens in the exceptional cases for $m \geq 3$?}

It is tempting to propose the following conjecture.  

\begin{conj}
If $P$ is an exceptional minuscule poset, then the triple $(J(P \times [m]), J(P \times [m];q), \Psi)$ exhibits the cyclic sieving phenomenon for all positive integers $m$.  
\end{conj}

\section{Acknowledgments}

This research was undertaken at the University of Minnesota, Twin Cities, under the direction of Profs. Victor Reiner, Gregg Musiker, and Pavlo Pylyavsky and with the financial support of the US National Science Foundation via grant DMS-1001933.  It is the authors' pleasure to extend their gratitude first and foremost to Prof. Reiner, not only for his extraordinary leadership of the REU (Research Experiences for Undergraduates) program hosted by the University of Minnesota, but also for his dedicated mentorship while this project was in progress and his continued support when it came time to ready the results for eventual publication.  The authors would also like to thank Jessica Striker and Nathan Williams for helpful conversations and a fruitful exchange of early drafts, Kevin Dilks for helpful conversations and the use of his Maple code for computations, and the two anonymous referees for all their constructive suggestions.  

\section{Appendix}

In this section, we choose six example root systems, one of each classical type as well as $E_6$ and $E_7$, and we present an illustration of every minuscule heap arising from a Lie algebra whose root system is among these.  The purpose is to provide an indication of all the possible shapes a minuscule heap may take on.      

\appendix
\section{The Case $A_n$}

For root systems of the form $A_n$, let $\alpha_j = \epsilon_{j+1}-\epsilon_j$ for all $1 \leq j \leq n$.  The possible minuscule weights are $\omega_1, \omega_2, \ldots, \omega_n$ (in other words, each fundamental weight may be minuscule), and the minuscule heaps arising from $A_4$ appear in Figure~\ref{figa}.    

\begin{figure}[htp]
\begin{center}
\subfigure[]{\includegraphics[scale = 0.5]{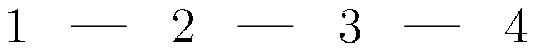} } \hspace{0.2in}
\subfigure[]{\includegraphics[scale = 0.5]{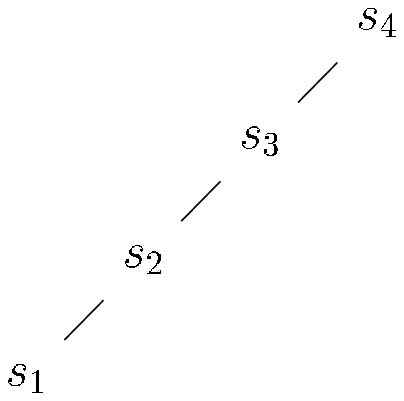} } \hspace{0.2in}
\subfigure[]{\includegraphics[scale = 0.5]{davidA4Two} } \hspace{0.2in}
\subfigure[]{\includegraphics[scale = 0.5]{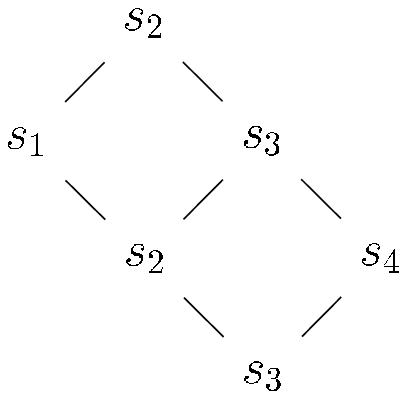} } \hspace{0.2in}
\subfigure[]{\includegraphics[scale = 0.5]{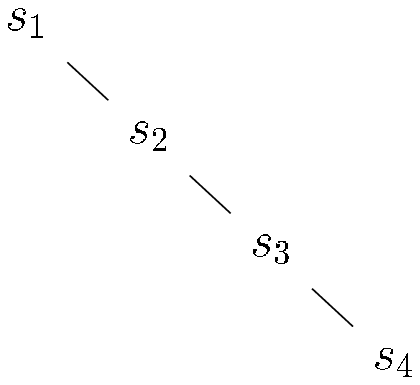} } 
\caption{Left to right: (a) the Dynkin diagram for root system $A_4$, (b) the heap $P_{w_0^J}$ for minuscule weight $\omega_1$, (c) the heap $P_{w_0^J}$ for minuscule weight $\omega_2$, (d) the heap $P_{w_0^J}$ for minuscule weight $\omega_3$, and (e) the heap $P_{w_0^J}$ for minuscule weight $\omega_4$.  }
\label{figa}
\end{center}
\vspace{0in}
\end{figure}

\section{The Case $B_n$}

For root systems of the form $B_n$, let $\alpha_1 = \epsilon_1$, and let $\alpha_j = \epsilon_{j}-\epsilon_{j-1}$ for all $2 \leq j \leq n$.  The only possible minuscule weight is $\omega_1$, and the minuscule heap arising from $B_4$ appears in Figure~\ref{figb}.  

\begin{figure}[htp]
\begin{center}
\subfigure[]{\includegraphics[scale = 0.6]{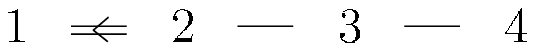} } \hspace{0.5in}
\subfigure[]{\includegraphics[scale = 0.6]{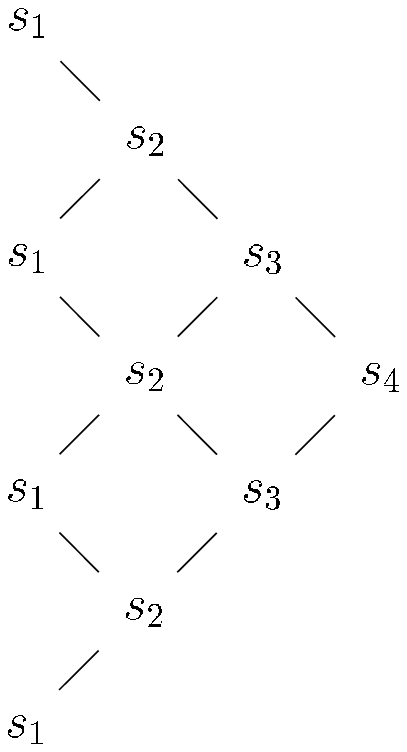} }
\caption{Left to right: (a) the Dynkin diagram for root system $B_4$ and (b) the heap $P_{w_0^J}$ for minuscule weight $\omega_1$.  }
\label{figb}
\end{center}
\vspace{0in}
\end{figure}

\section{The Case $C_n$}

For root systems of the form $C_n$, let $\alpha_1 = 2\epsilon_1$, and let $\alpha_j = \epsilon_{j}-\epsilon_{j-1}$ for all $2 \leq j \leq n$.  The only possible minuscule weight is $\omega_n$, and the minuscule heap arising from $C_5$ appears in Figure~\ref{figc}.  

\begin{figure}[htp]
\begin{center}
\subfigure[]{\includegraphics[scale = 0.6]{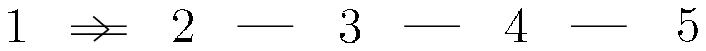} } \hspace{0.5in}
\subfigure[]{\includegraphics[scale = 0.6]{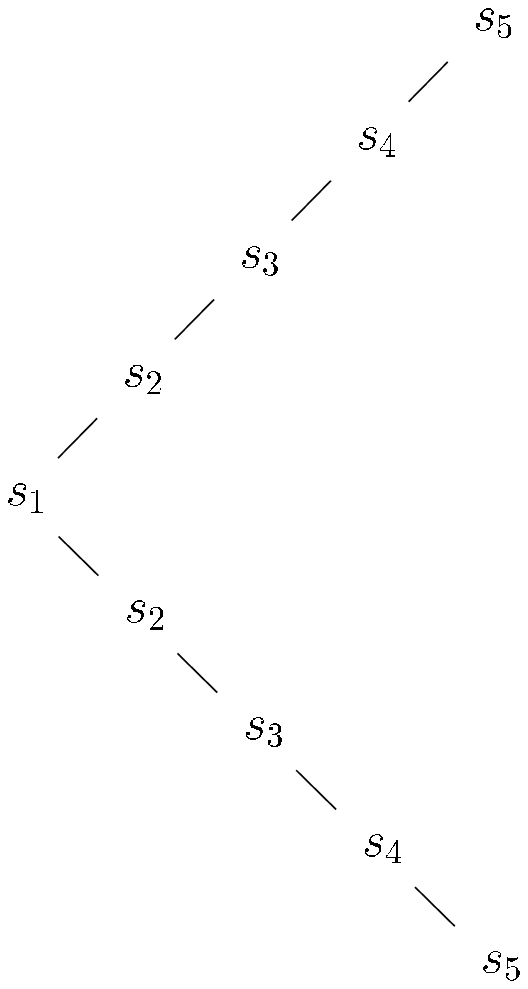} } 
\caption{Left to right: (a) the Dynkin diagram for root system $C_5$ and (b) the heap $P_{w_0^J}$ for minuscule weight $\omega_5$.  }
\label{figc}
\end{center}
\vspace{0in}
\end{figure}

\section{The Case $D_n$}

For root systems of the form $D_n$, let $\alpha_1 = \epsilon_1 + \epsilon_2$, and let $\alpha_j = \epsilon_{j}-\epsilon_{j-1}$ for all $2 \leq j \leq n$.  The only possible minuscule weights are $\omega_1$, $\omega_2$, and $\omega_n$, and the minuscule heaps arising from $D_5$ appear in Figure~\ref{figd1}.  

\begin{figure}[htp]
\begin{center}
\subfigure[]{\includegraphics[scale = 0.5]{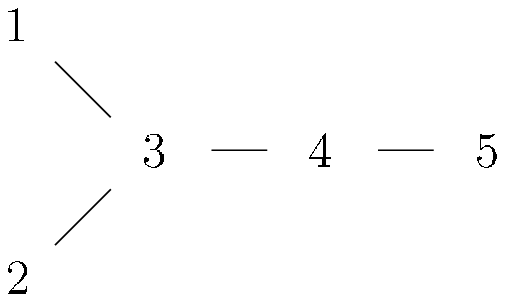} } \hspace{0.3in}
\subfigure[]{\includegraphics[scale = 0.5]{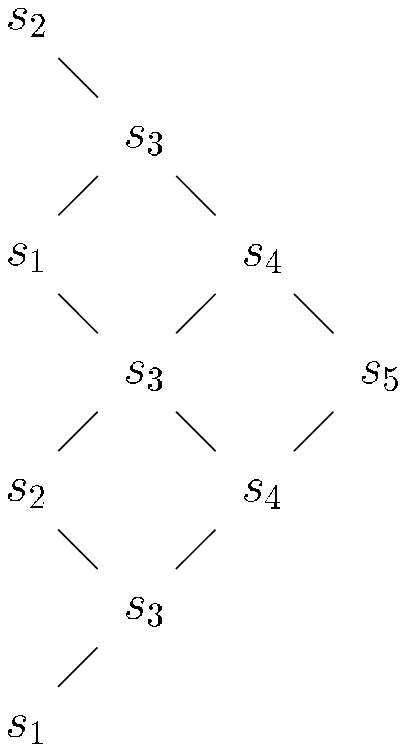} } \hspace{0.3in}
\subfigure[]{\includegraphics[scale = 0.5]{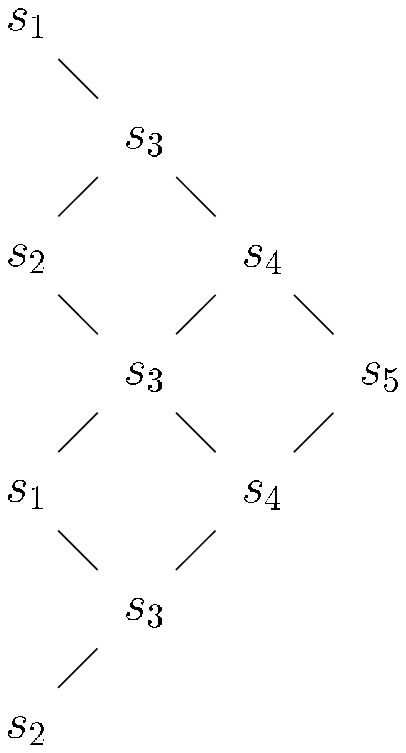} } \hspace{0.3in} 
\subfigure[]{\includegraphics[scale = 0.5]{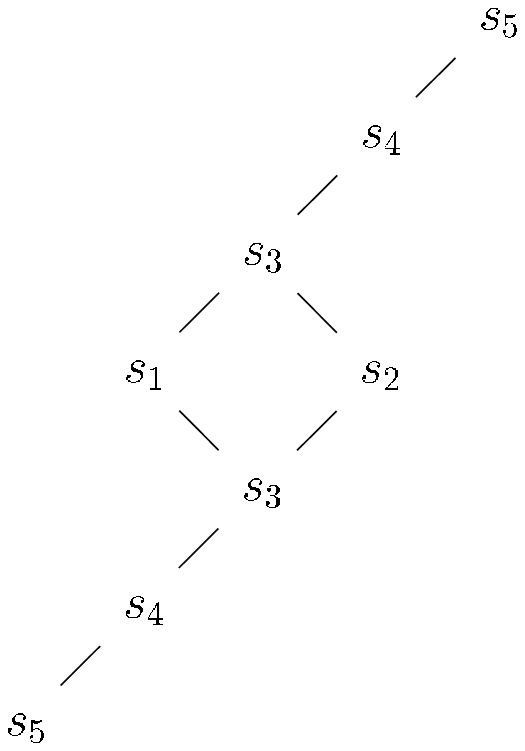} } 
\caption{Left to right: (a) the Dynkin diagram for root system $D_5$, (b) the heap $P_{w_0^J}$ for minuscule weight $\omega_1$, (c) the heap $P_{w_0^J}$ for minuscule weight $\omega_2$, and (d) the heap $P_{w_0^J}$ for minuscule weight $\omega_5$.  }
\label{figd1}
\end{center}
\vspace{0in}
\end{figure}

\section{The Exceptional Cases}
For root systems $E_6$ and $E_7$, let the simple roots be chosen to obey the relationships depicted in the Dynkin diagrams in Figures~\ref{fige6} and \ref{fige7}, respectively.  For the case $E_6$, the only possible minuscule weights are $\omega_1$ and $\omega_6$, and the corresponding heaps appear in Figure~\ref{fige6}.  For the case $E_7$, the only possible minuscule weight is $\omega_7$, and the corresponding heap appears in Figure~\ref{fige7}.    

\begin{figure}[htp]
\begin{center}
\subfigure[]{\includegraphics[scale = 0.5]{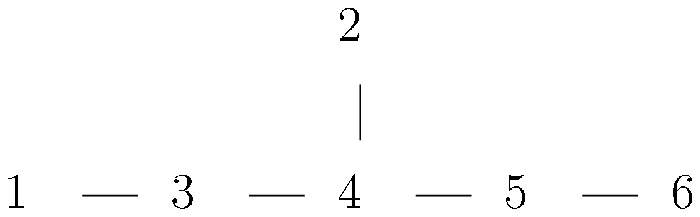} } \hspace{0.3in}
\subfigure[]{\includegraphics[scale = 0.5]{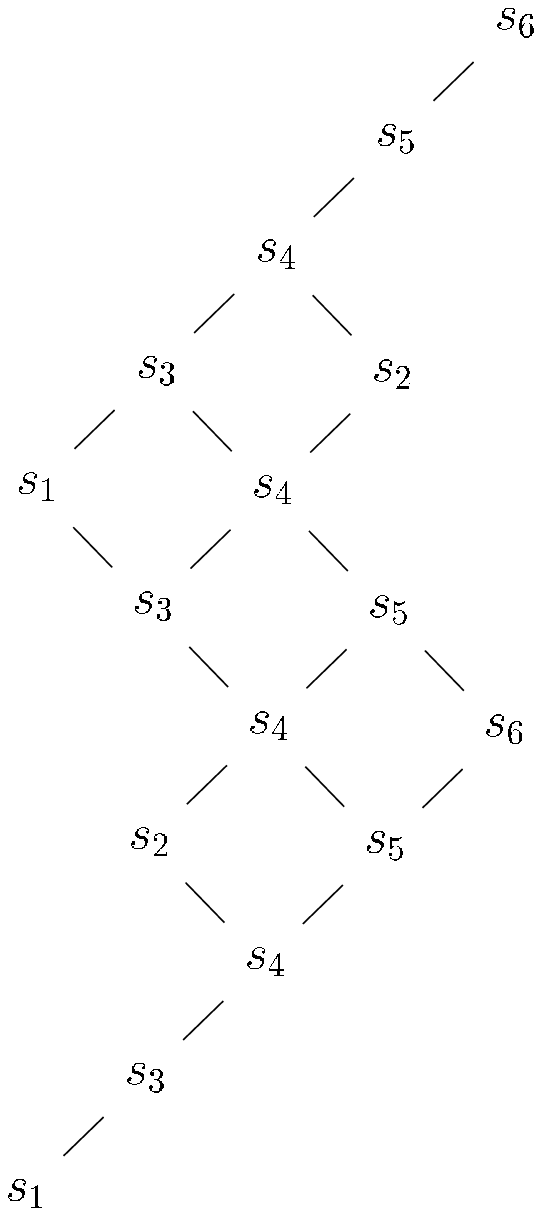} } \hspace{0.3in}
\subfigure[]{\includegraphics[scale = 0.5]{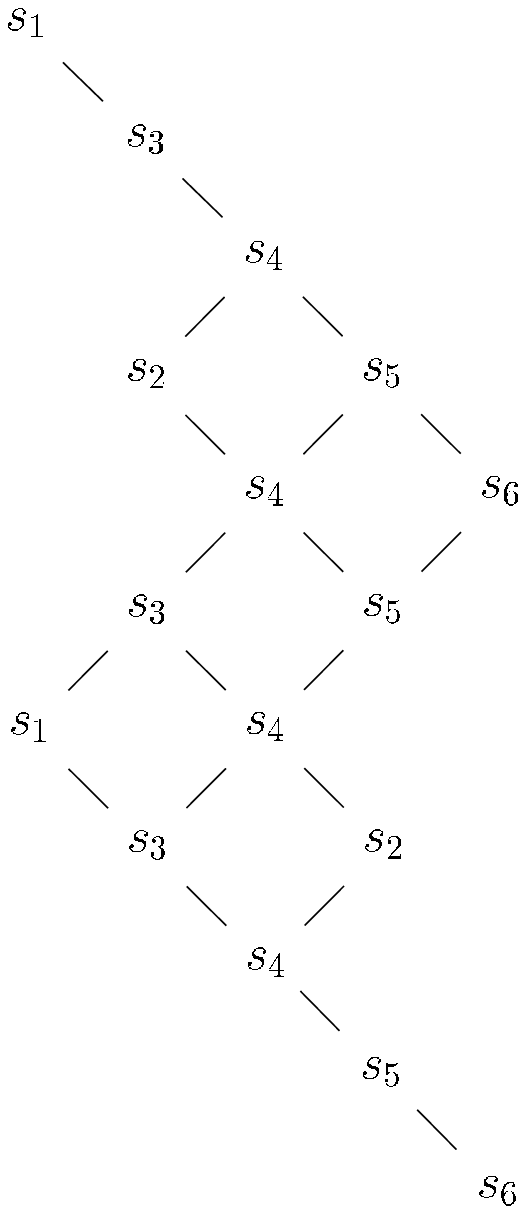} }
\caption{Left to right: (a) the Dynkin diagram for root system $E_6$, (b) the heap $P_{w_0^J}$ for minuscule weight $\omega_1$, and (c) the heap $P_{w_0^J}$ for minuscule weight $\omega_6$.  }
\label{fige6}
\end{center}
\vspace{0in}
\end{figure}

\begin{figure}[htp]
\begin{center}
\subfigure[]{\includegraphics[scale = 0.6]{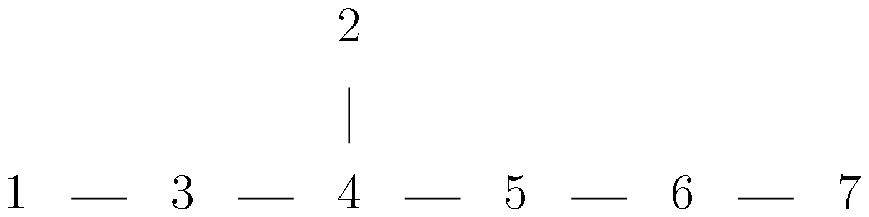} } \hspace{0.5in}
\subfigure[]{\includegraphics[scale = 0.6]{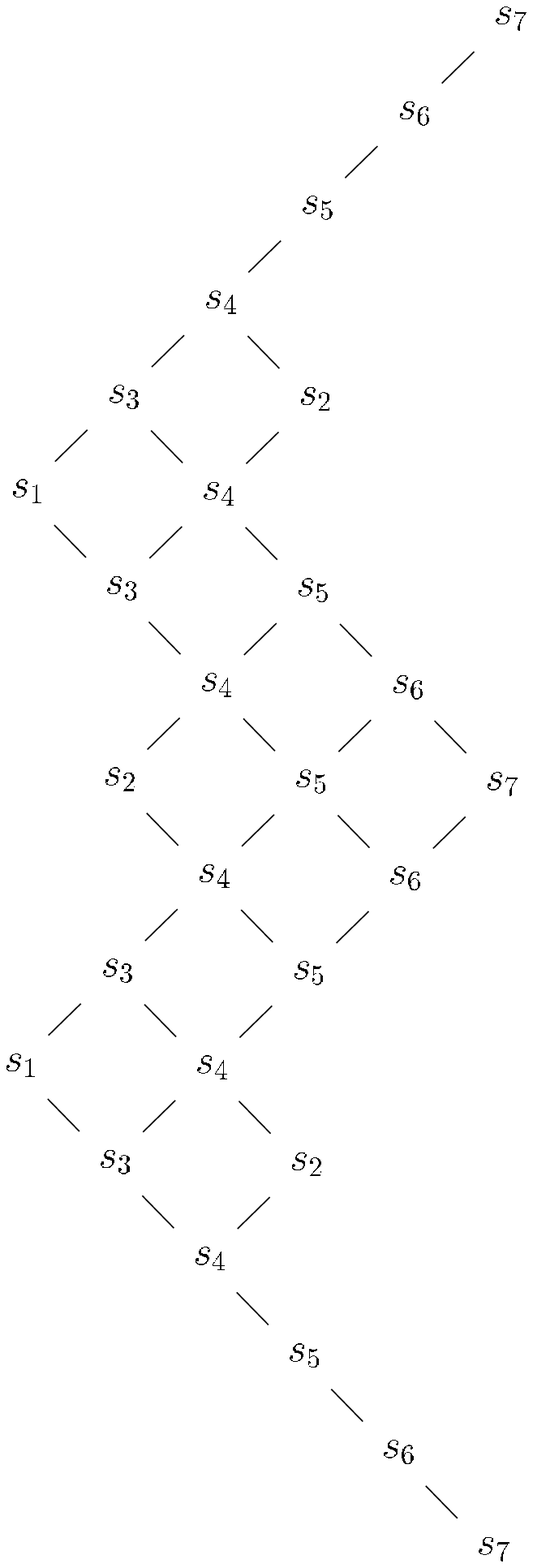} } 
\caption{Left to right: (a) the Dynkin diagram for root system $E_7$ and (b) the heap $P_{w_0^J}$ for minuscule weight $\omega_7$.  }
\label{fige7}
\end{center}
\vspace{0in}
\end{figure}

\end{document}